\documentclass[a4paper,10pt]{article}
\usepackage{indentfirst,amsmath,amsfonts,amstext,amssymb,amscd,bezier,amsthm}
\usepackage[latin1]{inputenc}
\usepackage[dvips]{graphicx}
\usepackage[latin1]{inputenc}
\usepackage{setspace}
\usepackage{graphicx}
\usepackage{tabularx}
\usepackage{longtable}
\usepackage[usenames,dvipsnames]{pstricks}
\usepackage{epsfig}
\usepackage{pst-grad}
\usepackage{pst-plot}
\usepackage[all]{xy}
\usepackage{hyperref}
\usepackage[T1]{fontenc}
\usepackage{ae,aecompl}
\usepackage{indentfirst}
\usepackage{xspace}
\usepackage{color}
\usepackage{psfrag}
\usepackage{mathrsfs}
\usepackage{upgreek}                            
\usepackage{makeidx}                            
\usepackage{fancyhdr}
\usepackage{fancybox}
\usepackage[rm]{titlesec}
\usepackage{float}   
\usepackage{enumerate}     
\input xy
\xyoption{all}
\newcommand{\p}{\mathbb{P}}

\newcommand{\F}{\mathbb{F}}

\newcommand{\PGL}{\mbox{\rm PGL}}
\newcommand{\lcm}{\mbox{\rm lcm}}
\newcommand{\PSL}{\mbox{\rm PSL}}
\newcommand{\PGU}{\mbox{\rm PGU}}
\newcommand{\GL}{\mbox{\rm GL}}

\newcommand{\ff}{\mathcal{F}}

\newcommand{\fq}{\mathbb{F}_q}
\newcommand{\fqq}{\mathbb{F}_{q^2}}

\newcommand{\kk}{\mathbb K}

\newcommand{\xx}{\mathcal X}

\newcommand{\cc}{\mathcal C}

\newcommand{\aut}{\operatorname{Aut}}
\newcommand{\tr}{\operatorname{Tr}}
\newcommand{\nr}{\operatorname{N}}
\newcommand{\gal}{\operatorname{Gal}}

\def\div{\mbox{\rm div}}

\theoremstyle{plain}
\newtheorem{thm}{Theorem}[section]

\newtheorem{prop}[thm]{Proposition}
\newtheorem{lem}[thm]{Lemma}

\newtheorem{rem}[thm]{Remark}

 \font\numberfont= pzcmi scaled
3000
\titleformat{\chapter}[display]
  {\normalfont\Large % roman
  }
  {%\titlerule[3pt]%
   \filright
   \rule[32pt]{.7\linewidth}{4pt}
   \hspace{-8pt}
   \shadowbox{
   \begin{minipage}{.15\linewidth}
     \begin{center}
          \textsl{\bf {\large \chaptertitlename}}\\
       \vspace{1ex}
       {\bf {\numberfont \thechapter}}\\
       \vspace{1ex}
     \end{center}
   \end{minipage}}
  }
  {-10pt}
  {\filcenter
           \sl
           \bf
              \Huge
     }
  [\vspace{-1cm}\hfill\rule{.8\textwidth}{0.5pt}\\
\vskip-2.8ex\hfill\rule{.7\textwidth}{4pt}\vspace*{-1ex}]
\titlespacing{\chapter}{0pt}{*4}{*1}

\titleformat{\section}[block]
{\normalfont\bfseries} {\thesection}{0.5em}{}

\titleformat{\subsection}[block]
{\normalfont\large\bfseries} {\thesubsection}{0.5em}{}

\makeindex                                      % comando para fazer o �ndice remissivo
\makeatletter
\numberwithin{equation}{section}

\begin{document}

\title{On generalizations of Fermat curves over finite fields and their automorphisms}

\author{\textbf{Nazar Arakelian} \\
 \small{CMCC, Universidade Federal do ABC, Santo Andr\'e, Brazil} \\
 \textbf{Pietro Speziali} \\
 \small{ Univesit\`a Degli Studi della Basilicata, Potenza, Italy}}
% \date{}

\maketitle
\begin{abstract}
Let $\xx$ be an irreducible algebraic curve defined over a finite field $\fq$ of characteristic $p>2$. Assume that the $\fq$-automorphism group of $\xx$ admits as an automorphism group the direct product of two cyclic groups $C_m$ and $C_n$ of orders $m$ and $n$ prime to $p$ such that both quotient curves $\xx/C_n$ and $\xx/C_m$ are rational. In this paper, we provide a complete classification of such curves, as well as a characterization of their full automorphism groups.

\end{abstract}

\section{Introduction} \label{intro}
One of the leading problems of algebraic geometry is the classification of algebraic varieties. As most leading problems, it is largely unsolved. This holds true even if we restrict ourselves to $1$-dimensional varieties, that is, algebraic curves. The essential tool in pursuing the goal of classifying (projective, nonsingular, geometrically irreducible, algebraic) curves is the study of their birational invariants, such as their genus and automorphism group.
Since a general curve has trivial automorphism group, any curve $\xx$ with an automorphism group $\aut(\xx) \neq \{1\}$ is of particular interest. Further, curves equipped  with a large automorphism group have a rich and interesting geometry. 
When the characteristic of the ground field $\mathbb{K}$ is some prime $p > 0$, several exceptions to the classical Hurwitz bound for the order of $\aut(\xx)$ are found, yielding classes of curves with particularly interesting properties. However, even in such exceptional cases, the automorphism group alone is not enough to characterize a curve, since Madden and Valentini \cite{MV} proved that for  for any finite group $G$  there exists infinitely many  non-isomorphic algebraic curves whose full automorphism group is isomorphic to $G$. Remarkably, in some cases it is possible to characterize a curve $\xx$ in terms of its automorphism group and genus. This happens for instance for the Hermitian curve, the Deligne-Lusztig-Suzuki curve and the Artin-Mumford curve; see  \cite{AK, HKT, RS}. 

Following this idea, one may ask which curves have a certain group $G$ as a subgroup of $\aut(\xx)$ with some extra condition on the action of  $G$ on the points of $\xx$. The classification problem becomes even more challenging when considering curves defined over some finite field $\fq$ of order $q=p^h$. This case is also of interest in view of applications to Coding Theory and Finite Geometry.  

In this paper, we classify all curves $\xx$ defined over a finite field $\fq$ of characteristic $p >2$ satisfying the following property: 

\begin{itemize}
\item [(P)] The $\fq$-automorphism group $\aut_{\fq}(\xx)$ of $\xx$ contains a subgroup $G=C_n \times C_m$, where $C_i$ denotes a cyclic group of order $i$ prime to $p$, such that $\max\{n,m\}>2$ and both quotient curves $\xx/C_n$ and $\xx/C_m$ are rational. 
\end{itemize}

If $n = m$ and  the $G$-short orbits are $\fq$-rational (that is, preserved by the $\fq$-Frobenius automorphism $\Phi_q$), a curve satisfying (P) is  a generalized Fermat curve  as introduced by Fanali and Giulietti in \cite{FGGFC}.  Thus, we will sometimes refer to a curve satisfying (P) as a generalized Fermat curve. It should be noted that the same term is used in the literature to describe similar yet rather different curves; see \cite{FGGFC, {HKLP}}.%add ref. 

As our main result, we provide the complete classification of the curves satisfying (P) and their full automorphism groups. 

%%%%%%%%%%%%%%%%%%%%%%%%%%%%%%%%%%%%%%%%%%%%%%%%%%%%%%%%%%%%%%%%%%%%%%%%%%%%%%%%%%%%%%%%%%%%%%%%%%%%%%%%%%%%%%%%%%%%%%%%%%%%%%%%%%%%%%%%%%%%%%%%%%%%%%%%%%%%%%%%%%%%%%%%%%%%%%%%%

\section{Background and preliminary results}\label{back}

Our notation and terminology are standard. Well-known references for the theory of curves and algebraic function fields are \cite{HKT} and \cite{stbook}. Let $\xx$ be a curve defined over some finite field $\fq$ of size $q = p^h$ for some prime $p$; then $\xx$ is viewed as a curve over the algebraic closure $\mathbb{K}$ of $\fq$. We denote by $\mathbb{K}(\xx)$ the function field of $\xx$. By a point $P \in \xx$ we mean a point in a nonsingular model of $\xx$; in this way, we have a one-to-one correspondence between points of $\xx$ and places of $\kk(\xx)$. Let $\aut_\kk(\xx)$ denote the full automorphism group of $\xx$. For a subgroup $S$ of $\aut_\kk(\xx)$, we denote by $\kk(\xx)^S$ the fixed field of $S$. A nonsingular model $\bar{\xx}$ of  $\kk(\xx)^S$ is referred as the quotient curve of $\xx$ by $S$ and denoted by $\xx/S$. Note that $\xx/S$ is defined up to birational equivalence. The field extension $\kk(\xx):\kk(\xx)^S$ is Galois with Galois group $S$. For a point $P \in \xx$, $S(P)$ is the orbit of $P$ under the action of $S$ on $\xx$ seen as a point-set. The orbit $S(P)$ is said to be long if $|S(P)| = |S|$, short otherwise. There is a one-to-one correspondence between short orbits and ramified points in the extension $\kk(\xx):\kk(\xx)^S$. It might happen that $S$ has no short orbits; if this is the case, the cover $\xx \rightarrow \xx/S$ (or equivalently, the extension $\kk(\xx):\kk(\xx)^S$) is unramified. On the other hand, $S$ has a finite number of short orbits. 

For $P \in \xx$, the subgroup $S_P$ of $S$ consisting of all elements of $S$ fixing $P$ is called the stabilizer of $P$ in $S$.  For a non-negative integer $i$, the $i$-th ramification group of $\xx$ at $P$ is denoted by $S_P^{(i)}$, and defined by
$$
S_P^{(i)}=\{\sigma \ | \ v_P(\sigma(t)-t)\geq i+1, \sigma \in S_P\}, 
$$
 where $t$ is a local parameter at $P$ and $v_P$ is the respective discrete valuation. Here $S_P=S_P^{(0)}$. Furthermore, $S_P^{(1)}$ is a normal $p$-subgroup of $S_P^{(0)}$, and the factor group $S_P^{(0)}/S_P^{(1)}$ is cyclic of order prime to $p$; see e.g. \cite[Theorem 11.74]{HKT}. In particular, if $S_P$ is a $p$-group, then $S_P=S_P^{(0)}=S_P^{(1)}$. 

Let $g$ and $\bar{g}$ be the genus of $\xx$ and $\bar{\xx}=\xx/S$, respectively. The Riemann-Hurwitz genus formula is 
\begin{equation}\label{rhg}
2g-2=|S|(2\bar{g}-2)+\sum_{P \in \xx}\sum_{i \geq 0}\big(|S_P^{(i)}|-1\big);
\end{equation}
see \cite[Theorem 11.72]{HKT}.
 If $\ell_1,\ldots,\ell_k$  are the sizes of the short orbits of $S$, then (\ref{rhg}) yields
\begin{equation}\label{rhso}
2g-2 \geq |S|(2\bar{g}-2)+\sum_{\nu=1}^{k} \big(|S|-\ell_\nu\big),
\end{equation}
and equality holds if $\gcd(|S_P|,p)=1$ for all $P \in \xx$; see \cite[Theorem 11.57 and Remark 11.61]{HKT}.

The following result (see \cite[Proposition 1]{Ko2}) will be used in Section \ref{full}.

\begin{prop}[Kontogeorgis]\label{konto1}
Let $\F_0$ be a rational function field over $\kk$. Suppose that a cyclic extension $\F$ of $\F_0$ is completely ramified  at $s$ places and $r = |\gal(\F:\F_0)|$. If $2r < s$ then $\gal(\F : \F_0)$ is normal on  the full automorphism group $\aut_{\kk}(\F)$ of $\F$ .
\end{prop}

Let $\Phi_q:\xx \rightarrow \xx$ denote the $\fq$-Frobenius map. An automorphism $\sigma \in \aut_\kk(\xx)$ is said to be $\fq$-rational if it commutes with $\Phi_q$. A subgroup $S$ of $\aut_\kk(\xx)$ is $\fq$-rational if every element of $S$ commutes with $\Phi_q$. The subgroup of $\aut_{\kk}(\xx)$ consisting of all $\fq$-rational automorphisms is called the $\fq$-automorphism group of $\xx$, and it is denoted by $\aut_{\fq}(\xx)$. Note that $\xx/S$ is defined over $\fq$ for all $S< \aut_{\fq}(\xx)$.

%%%%%%%%%%%%%%%%%%%%%%%%%%%%%%%%%%%%%%%%%%%%%%%%%%%%%%%%%%%%%%%%%%%%%%%%%%%%%%%%%%%%%%%%%%%%%%%%%%%%%%%%%%%%%%%%%%%%%%%%%%%%%%%%%%%%%%%%%%%%%%%%%%%%%%%%%%%%%%%%%%%%%%%%%%%%%%%%

\section{Cyclic subcovers of the projective line}

The function field $\kk(\cc)$ of a rational curve $\cc$ is such that $\mathbb{F} = \kk(x)$ for some rational function $x \in \kk(\cc)$. Since $\ff$ is birationally equivalent to $\mathbb{P}^1(\kk)$, we have that $\aut_\kk(\cc) \cong \PGL(2,\kk)$. If $\cc$ is defined over $\fq$, then  $\aut_{\fq}(\cc) \cong \PGL(2,q)$.  
We are interested in  quotients of the projective line arising from tame cyclic subgroups of $\PGL(2,q)$; 
 by Dickson's Hauptsatz \cite[Theorem 3]{VM}, such groups have order $k$ a divisor of $q\pm 1$.

 Let $\F$ be a rational function field over $\fq$. Consider a cyclic extension $\F:\F^{'}$, where $\F^{'}$ is a subfield of $\F$ defined over $\fq$. By L\"uroth's Theorem, $\F^{'}$ is  rational as well; see \cite[Theorem 3.5.9]{stbook}. 

\begin{prop}\label{n|q-1}
Let $\F$ be a rational function field over $\fq$, where $q=p^h$, with $p>2$. Let $\F^{'}$ be a subfield of $\F$ such that the extension $\F : \F^{'}$ is cyclic of degree $n$ prime to $p$, with $n|q-1$. Assume that the ramified places of $\F : \F^{'}$ are $\fq$-rational. Then there exists $x \in \F$ such that $\F=\fq(x)$ and $\F^{'}=\fq(x^n)$.
\end{prop}
\begin{proof}
Let $\sigma \in \aut_{\fq}(\F)$ be a generator of a cyclic subgroup of $\aut_{\fq}(\F)$ of order $n$. Since $n|q-1$, then $\fq$ has a $n$-th primitive root of the unity $\zeta$. There exists $x \in \F$ with exactly one zero (and one pole) defined over $\fq$ such that $\sigma(x) = \zeta x$. Then $\sigma(x^n) = x^n$ holds. Clearly $\F = \fq(x)$, whence our assertion follows. 
\end{proof}

Now let us consider the case $[\F:\F^{'}]=n | q+1$. So let $x \in \F$ such that $\F=\fq(x)$, and let $G$ be a cyclic subgroup of $\aut_{\fq}(\fq(x))$ with $|G|=n$ such that $n|q+1$. Fix a nonsquare element $s \in \fq$ and define $\mathbb{F}_{q^2}$ as an extension $\fq(i)$ of $\fq$ with $i \in \F_{q^2}$ such that $i^2=s$. 
Then $\F_{q^2}=\{a+bi \  | \  a,b \in \fq\}$.

For $\alpha=a+bi \in \F_{q^2}^{*}$, set
$$M_\alpha=\left(
           \begin{array}{cc}
             a & sb \\
             b & a \\
           \end{array}
         \right) \in \GL(2,q).
$$
The map $\alpha \mapsto M_\alpha$ is a monomorphism from the multiplicative group $\F_{q^2}^{*}$ to $\GL(2,q)$. Let $\lambda \in \F_{q^2}$ be a primitive $(2n)$-th root of the unity. Then the subgroup $\langle M_\lambda\rangle$ of $\GL(2,q)$ is cyclic of order $2n$. The natural group homomorphism $\varphi: \GL(2,q)\rightarrow \PGL(2,q)$ is surjective
and $\ker \varphi$ consists of all scalar matrices. Via a simple computation, one can show that $\ker \varphi \cap \langle M_\lambda\rangle = \{M_1,M_{-1}\}$. Hence $\varphi$ maps $\langle M_\lambda\rangle$ to a subgroup $C$ of $\PGL(2,q)$ of order $n$. Note that the fixed points of $C$ are $(i:1)$ and $(-i:1)$. From the classification of subgroups of $\PGL(2,q)$ we know that there exists only one class of cyclic subgroups of order $n$ fixing points not defined over $\fq$ \footnote{The only situation in which a cyclic subgroup of $\PGL(2,q)$ of order $n|q+1$ fixes an $\fq$-rational point is $n=2$. }. Therefore, we have the following result.

\begin{prop}\label{repr}
Let $G$ be a cyclic subgroup of $\aut_{\fq}(\fq(x))$ with $|G|=n$ fixing no $\fq$-rartional place such that $n|q+1$. Then, up to conjugacy, $G=\langle \tau \rangle$, where
\begin{equation}\label{auto}
\tau(x)=\frac{ux+sv}{vx+u},
\end{equation}
with $u+iv \in \F_{q^2}$ being a primitive $(2n)$-th root of the unity.
\end{prop}

\begin{prop}\label{ffq+1}
Let $\F$ be a rational function field over $\fq$, where $q=p^h$, with $p>2$. Let $\F^{'}$ be a subfield of $\F$ defined over $\fq$ such that the extension $\F : \F^{'}$ is cyclic of order $n$ prime to $p$ with no ramified $\fq$-rational place, with $n|q+1$. Then there exists $x \in \F$ such that $\F=\fq(x)$ and $\F^{'} = \fq(z)$ with $z$ given by
\begin{equation}
z = \frac{i[(x+i)^n-(x-i)^n]}{(x+i)^n+(x-i)^n}.
\end{equation}
\end{prop}
\begin{proof}
Let $\tau \in \aut_{\fq}(\F)$ be such that $\F^{'}=\F^{\langle \tau \rangle}$. Let $x \in \F$ such that $\F=\fq(x)$ and $\tau$ is defined on $\F$ by (\ref{auto}). 
 Consider the $\fqq$-rational function $h(x)$ given by
\begin{equation}
h(x) = \frac{(x-i)^n}{(x+i)^n}.
\end{equation}
 
A straightforward computation shows that $\tau(h(x)) = \frac{(u-iv)^n(x-i)^n}{(u+iv)^n(x+i)^n} = h(x)$ as $(u-iv)^n = (u+iv)^{qn} = (u+iv)^n$. Then $\fqq(h(x)) \subseteq \fqq(x)^{<\tau>}$.  From $n \leq [\fqq(x): \fqq(h(x))] \leq n$ we get $\fqq(h(x)) = \fqq(x)^{<\tau>}$. Let 

\begin{equation}
z = i \cdot \frac{h(x)-1}{h(x)+1} = \frac{i[(x-i)^n-(x+i)^n]}{(x-i)^n+(x+i)^n}. 
\end{equation}
Then  $z  \in \fqq(h(x)) \cap \fq(x)$ with $[\fqq(h(x)) : \fqq(z)] = 1$, that is, $\F^{'} = \fq(z)$.
\end{proof}

\begin{rem}
Let $\F$ and $\F^{'}$ as in Proposition \ref{ffq+1}. We may assume that $\F=\fq(x)$ and $\F^{'}=\F^{\langle \tau \rangle}$, with $\tau$ defined as in (\ref{auto}). It can be shown that $\F^{'}=\fq(\tr(x))$, where $\tr:\F \rightarrow \F^{'}$ is the trace map of the extension $\F:\F^{'}$. In the same way, under the hypothesis of Lemma \ref{n|q-1}, let $\nr:\F \rightarrow \F^{'}$ denote the norm map of the extension $\F:\F^{'}$. It can be shown that $\F=\fq(x)$ and $\F^{'}=\fq(\nr(x))$ for some $x \in \F$. 

\end{rem}

We finish this section with a direct consequence of \cite[Theorem 2]{VM}.

\begin{lem}\label{tec}
Let $\mathcal{Y}$ be a rational curve defined over $\fq$. Suppose that $P$ is fixed by a subgroup $C \subset \aut_{\kk}(\mathcal{Y})$ of order $n$ such that $\gcd(p,n)=1$. If $P$ is $\fq$-rational, then $C$ is cyclic and $n|q-1$. If $P$ is not $\fq$-rational, then it is $\mathbb{F}_{q^2}$-rational and $n|q+1$.
\end{lem}

%%%%%%%%%%%%%%%%%%%%%%%%%%%%%%%%%%%%%%%%%%%%%%%%%%%%%%%%%%%%%%%%%%%%%%%%%%%%%%%%%%%%%%%%%%%%%%%%%%%%%%%%%%%%%%%%%%%%%%%%%%%%%%%%%%%%%%%%%%%%%%%%%%%%%%%%%%%%%%%%%%%%%%%%%%%%%%%%%%%%

\section{Geometric properties of generalized Fermat curves}\label{geometric}

In this section, some geometric features of a curve $\xx$ satisfying property (P) are described. In particular,  some results from \cite[Section 3]{FGGFC} are generalized. %Here, the classification of the subgroups of $\PGL(2,q)$ and standard tools of function field theory are exploited.

\begin{lem}\label{group1}
Let $\xx$ be a curve satisfying (P). Then  $m| q - 1$ or $m|q+1$, and the same holds for $n$.
\end{lem}
\begin{proof}
Since $C_m$ normalizes $C_n$, there is a subgroup $\tilde{C}_m$ of $\aut_{\fq}(\xx/C_n)$ such that $\tilde{C}_m \cong C_m$. By $\xx/C_n \cong \p^1(\fq)$, it follows that $\tilde{C}_m$ is isomorphic to a cyclic subgroup of $\PGL(2,q)$. The result follows from \cite[Theorem 3]{VM}.
\end{proof}

\begin{lem}\label{fiberp}
The function field $\fq(\xx)$ of $\xx$ is the compositum of $\fq(\xx/C_n)$ and $\fq(\xx/C_m)$.
\end{lem}
\begin{proof}
Set $\F=\fq(\xx/C_n) \cdot \fq(\xx/C_m)$. Since $\fq(\xx/C_n)=\fq(\xx)^{C_n}$ and $\fq(\xx/C_m)=\fq(\xx)^{C_m}$, then the extension $\fq(\xx):\F$ is Galois with Galois group 
$$
\gal(\fq(\xx):\F)=\gal(\fq(\xx):\fq(\xx)^{C_n})\cap\gal(\fq(\xx):\fq(\xx)^{C_m})=C_n \cap C_m=\{1\}.
$$
Therefore $\fq(\xx)=\F$.
\end{proof}

We now present the main result of this section.

\begin{prop}\label{genus}
Let $\xx$ be a curve of genus $g$ satisfying (P). Denote by $t$ the number of short orbits of $G$, and by $\ell_1, \ldots \ell_t$ their sizes. Then $m$ divides $q-1$ or $q+1$, $n$ divides $q-1$ or $q+1$, and one of the following holds:
\begin{itemize}
\item [(I)] $t=3$, $\ell_1=m$, $\ell_2=n$, $\ell_3=\gcd(m,n)$, and $g=\frac{mn-m-n-\gcd(m,n)+2}{2}$.
\item [(II)] $t=4$, $\ell_1=\ell_2=m$, $\ell_3=\ell_4=n$, and  $g=mn-m-n+1$.
\end{itemize} 
\end{prop}
\begin{proof}
It follows from Lemma \ref{group1} that $m$ (and $n$) divides $q-1$ or $q+1$. By \cite[Theorem 1]{VM}, there are two distinct points $P_1, P_2 \in \xx/G$ that are fully ramified in the cover $\xx/C_n \rightarrow \xx/G$, and the remaining points of $\xx/G$ split completely in $\xx/C_n$. Analogously, there are two distinct points $Q_1, Q_2 \in \xx/G$ (not necessarily distinct from $P_1$ and $P_2$) fully ramified in the cover  $\xx/C_m\rightarrow \xx/G$, with the remaining points of $\xx/G$ splitting completely in $\xx/C_m$. Since the cover $\xx \rightarrow \xx/G$ is tame, Lemma \ref{fiberp} and Abhyankar's Lemma (see e.g. \cite[Theorem 3.1.9]{stbook}) imply that the possible sizes of a nontrivial one-point stabilizer in $G$ are $m$, $n$ and $\lcm(m,n)$. In other words, the possible sizes of the short orbits of $G$ are $m$, $n$ and $\gcd(m,n)$. Moreover, it also follows from Abhyankar's Lemma that a point-set $\Omega \subset  \xx$ is a short orbit of $G$ if and only if $\Omega$ lie over $P_i$ or $Q_i$, $i=1,2$. In particular, $2 \leq t \leq 4$. Let $t_1, t_2,$ and $t_3$ be the number of short orbits of $G$ of size $n$, $m$ and $\gcd(m,n)$, respectively. Since  $\xx/G$ is rational, the Riemann-Hurwitz formula (\ref{rhso}) applied to the cover $\xx \rightarrow \xx/G$ yields
\begin{equation}\label{rhp}
2g-2=(t-2)mn-t_1n-t_2m-t_3\gcd(m,n).
\end{equation}

Suppose that $t=2$. Then $\{P_1,P_2\}=\{Q_1,Q_2\}$ and $t_1=t_2=0$. By (\ref{rhp}) we obtain that $\gcd(m,n)=1$ and $g=0$. Hence $\xx$ is a rational curve such that $\aut_{\fq}(\xx)$ has a subgroup isomorphic to $C_n \times C_m$ with $\gcd(m,n)=1$, which is not allowed by the classification of the subgroups of $\PGL(2,q)$ (\cite[Theorem 3]{VM}). Therefore, $t \in \{3,4\}$.

Assume that $t=3$. Then, without loss of generality, $P_2=Q_2$ and the short orbits of $G$ lying over $P_1$, $P_2$ and $Q_1$ have size $n$, $\gcd(m,n)$ and $m$, respectively. Thus from (\ref{rhp}) we have
$$
g=\frac{mn-m-n-\gcd(m,n)+2}{2}.
$$  
Finally, assume that $t=4$. Then $\{P_1,P_2\} \cap \{Q_1,Q_2\}=\emptyset$, the short orbits of $G$ lying over $P_1$ and $P_2$ have size $n$ and the short orbits of $G$ lying over $Q_1$ and $Q_2$ have size $m$. Hence, by $(\ref{rhp})$ we obtain $g=mn-m-n+1$.
\end{proof}

%%%%%%%%%%%%%%%%%%%%%%%%%%%%%%%%%%%%%%%%%%%%%%%%%%%%%%%%%%%%%%%%%%%%%%%%%%%%%%%%%%%%%%%%%%%%%%%%%%%%%%%%%%%%%%%%%%%%%%%%%%%%%%%%%%%%%%%%%%%%%%%%%%%%%%%%%%%%%%%%%%%%%%%%%%%%%%%%%%%

\section{Classification results}\label{classification}

Let us recall that $q=p^h$ with $p>2$ and %As before, in what follows both integers 
$m$ and $n$ divide $q \pm 1$. In case that $n$  (resp. $m$) divides $q+1$, we set the following notation. Fix a non-square $s \in \fq$ and choose a root $i$ of the polynomial $X^2-s$. Then $\F_{q^2}=\{a_0+ia_1 \ | \ a_0,a_1 \in \fq\}$. % Fix a $(2n)$-th (resp. $(2m)$-th) primitive root of the unity $\lambda=u+iv \in \F_{q^2}$ (resp. $\varepsilon=l+jw \in \F_{q^2}$). 
Our main result characterizes the curves satisfying property (P).

\begin{thm}\label{main}
Let $\xx$ be a  curve of genus $g$ defined over $\fq$ satisfying (P). Denote by $t$ the number of short orbits of $G$, and by $\ell_1, \ldots \ell_t$ their sizes. Then one of the following holds:
\begin{itemize}
\item[(a)] $t=3$, $\ell_1=m$, $\ell_2=n$, $\ell_3=\gcd(m,n)$, and $g=\frac{mn-m-n-\gcd(m,n)+2}{2}$. Furthermore, each short orbit of $G$ is preserved by the $\fq$-Frobenius map, both $n$ and $m$ divide $q-1$, and $\xx$ is $\fq$-birationally equivalent to the curve defined by
\begin{equation}\label{t=3}
aX^n+bY^m=1,
\end{equation}
where $ a, b \in \fq^{*}.$
\item[(b)] $t=4$, $\ell_1=\ell_2=m$, $\ell_3=\ell_4=n$, and  $g=mn-m-n+1$. Moreover, one of the following occurs:
\begin{itemize}
\item[(b1)] Each short orbit of $G$ is preserved by the $\fq$-Frobenius map, both $n$ and $m$ divide $q-1$, and $\xx$ is $\fq$-birationally equivalent to the curve defined by
\begin{equation}\label{t=4 1}
aX^nY^m+bX^n+cY^m=1, 
\end{equation}
where $a,b,c \in \fq$ with $c \neq \frac{a}{b}$ and $a \neq 0$. 
\item[(b2)] Only two short orbit of $G$ are preserved by the $\fq$-Frobenius map, without loss of generality $m|q-1$ and $n|q+1$ , and $\xx$ is $\fq$-birationally equivalent to the curve defined by
\begin{equation}\label{t=4 2}
\frac{aY^m+b}{cY^m+d} = \frac{i[(X+i)^n-(X-i)^n]}{(X+i)^n+(X-i)^n},
   \end{equation}
where $a,b,c,d \in \fq$.
\item[(b3)] $G$ has no short orbits preserved by the $\fq$-Frobenius map, both $n$ and $m$ divide $q+1$, and $\xx$ is $\fq$-birationally equivalent to the curve defined by
\begin{equation}\label{t=4 3}
\frac{[(ai+b)(X-i)^n+(b-ai)(X+i)^n][(Y-i)^m+(Y+i)^m]}{ i[(ci+d)(X-i)^n+(d-bi)(X+i)^n][(Y-i)^m-(Y+i)^m]} = 1,
\end{equation}
where $a,b,c,d \in \fq$. 
\end{itemize}
\end{itemize}
\end{thm}

The proof of Theorem \ref{main} will be obtained after a sequence of partial results. Henceforth, we denote by $\pi_1:\xx \rightarrow \xx/C_n$, $\pi_2:\xx \rightarrow \xx/C_m$ and $\pi:\xx \rightarrow \xx/G$ the natural projections of $\xx$ onto the quotient curves $\xx/C_n$, $\xx/C_m$ and $\xx/G$ respectively.  
We will make use of the following fact. 

\begin{lem}\label{fact}
Let $z \in \fq(\xx/C_n), z' \in \fq(\xx/C_m)$ be such that $\fq(\xx/C_n)^{C_m} = \fq(z)$ and $\fq(\xx/C_m)^{C_n} = \fq(z')$. Then there is $\tau \in \PGL(2,q)$ such that $z' = \tau(z)$.
\end{lem}
\begin{proof}
Clearly, $\fq(\xx)^G = (\fq(\xx)^{C_m})^{C_n} = (\fq(\xx)^{C_n})^{C_m}$. Then $\fq(z) = \fq(z')$. 
\end{proof}

\begin{lem}\label{so t3}
Let $\xx$ be a curve defined over $\fq$, where $q=p^h$ ($p>2$), satisfying (P). Assume that $G=C_n \times C_m$ has three short orbits in $\xx$. Then each short orbit of $G$ is preserved by the $\fq$-Frobenius map. Moreover, both $n$ and $m$ divide $q-1$.
\end{lem}
\begin{proof}
Recall that $\Phi_q$ denotes the $\fq$-Frobenius map. Since $G$ is defined over $\fq$, we have that
 $\Phi_q$ acts on the set of orbits of $G$.  
 As $\Phi_q$ is bijective, it acts on the set of short orbits of $G$. Furthermore, since $C_n$ and $C_m$ are defined over $\fq$, then  
 $\pi_i \circ \Phi_q=\Phi_q \circ \pi_i$ for $i\in \{1,2\}$. Set $\delta=\gcd(m,n)$ and let $\Omega_1=\{P_1^{1},\ldots,P_1^{n}\}$, $\Omega_2=\{P_2^{1},\ldots,P_2^{\delta}\}$ and $\Omega_3=\{P_3^{1},\ldots,P_3^{m}\}$ be the short orbits of $G$.
The Riemann-Hurwitz formula (\ref{rhso}) applied to the cover of curves $\xx \rightarrow \xx/C_m$ yields
\begin{equation}\label{rhpfqr}
n(m-1) +(m-\delta) = \sum_{\nu= 1}^k (m-\ell_\nu).
\end{equation}
 Since $|\Omega_1|=n$, the stabilizer in $G$ of a point $P_1^{i} \in \Omega_1$ has order $m$. 
Then, since $\pi_2(\Omega_2)$ and $\pi_2(\Omega_3)$ are over the only ramified points of $\xx/C_m \rightarrow \xx/G$, we conclude that $C_m$ fixes $\Omega_1$ elementwise. Also, from $C_m$ preserving the $\Omega_j$, we obtain that $\Omega_2$ forms a single orbit under $C_m$. From \eqref{rhpfqr}, $C_m$ acts semi-regularly on the other points of $\xx$. %Note that $\Omega_2$ is not a short orbit in such a cover if and only if $m = \delta$.
  Thus $\Phi_q(\sigma(P)) = \sigma(\Phi_q(P))$ for any $P \in \xx$ and $\sigma \in C_m$ imply $\phi_q(P_1^{i}) = \phi_q(P_1^{j})$, i.e. $\Omega_1$ is $\fq$-rational.  The same argument applied to $C_n$ shows that $\Omega_3$ (and consequently $\Omega_2$) is $\fq$-rational. Now $\pi_2(\Omega_3) \in  \xx / C_m$ is an $\fq$-rational point and it is fully ramified in the cover $\xx/C_m \rightarrow \xx/G$. Therefore, $n | q-1$ by Lemma \ref{tec}. Since in the proof we can interchange the roles $C_m$ and $C_n$, our claim follows.
	
\end{proof}

\begin{prop}\label{eq t3}
Let $\xx$ be a curve defined over $\fq$, where $q=p^h$ ($p>2$), satisfying (P). Assume that $G=C_n \times C_m$ has three short orbits in $\xx$. Then $\xx$ is $\fq$-birationally equivalent to a curve defined by $aX^n+bY^m=1$, with $a,b \in \fq^{*}$.
\end{prop}
\begin{proof}
First, both $n$ and $m$ divide $q-1$, and the short orbits of $G$ are $\fq$-rational, by Lemma \ref{so t3}. Thus Lemma \ref{n|q-1} implies that $\fq(\xx/C_n)=\fq(y)$, $\fq(\xx/C_m)=\fq(x)$ and $\fq(\xx/G)=\fq(y^m)=\fq(x^n)$, with $x,y \in \fq(\xx)$. Moreover, it follows from Lemma \ref{fiberp} that $\fq(\xx)=\fq(x,y)$. The extension $\fq(y):\fq(y^m)$ (resp. $\fq(x):\fq(x^n)$ ) has only two ramified points: the zero and the pole of $y^m$ (resp. $x^m$). Since each short orbit of $G$ lie over only one of this points, we may assume (without loss of generality) that $x^n$ and $y^m$ have a common pole and distinct zeros. Therefore $y^m=\alpha x^n+\beta$ for certain $\alpha, \beta \in \fq^{*}$. The result then follows from the irreducibility of the last equation.
\end{proof}

\begin{lem}\label{so t4}
Let $\xx$ be a curve defined over $\fq$, where $q=p^h$ ($p>2$), satisfying (P). Assume that $G=C_n \times C_m$ has four short orbits in $\xx$. Then one of the following holds:
\begin{itemize}
\item[(a)] Each short orbit of $G$ is preserved by the $\fq$-Frobenius map, and both $n$ and $m$ divide $q-1$.
\item[(b)] Only two short orbit of $G$ are preserved by the $\fq$-Frobenius map, $m|q-1$ and $n|q+1$ (or vice-versa).
\item[(c)] $G$ has no short orbits preserved by the $\fq$-Frobenius map, and both $n$ and $m$ divide $q+1$.
\end{itemize}
\end{lem}
\begin{proof}
Denote by $\Omega_\iota$ the short orbits of $G$, where $\iota \in \{1,2,3,4\}$. According to the proof of Proposition \ref{genus}, $\pi(\Omega_1)=P_1$, $\pi(\Omega_2)=P_2$, $\pi(\Omega_3)=Q_1$ and $\pi(\Omega_4)=Q_2$, with such points being pairwise distinct. Arguing as in the proof of Lemma \ref{so t3}, it can be shown that $\Phi_q$ preserves the point-sets $\Omega_1 \cup \Omega_2$  and $\Omega_3 \cup \Omega_4$.
Since $\Phi_q$ acts on the set of short orbits of $G$, we only have the following possibilities:
\begin{itemize}
\item[(a1)] Each $\Omega_\iota$ is preserved by $\Phi_q$. 
\item[(b1.1)] $\Phi_q$ preserves $\Omega_1$ and $\Omega_2$ and interchanges $\Omega_3$ and $\Omega_4$.
\item[(b1.2)] $\Phi_q$ preserves $\Omega_3$ and $\Omega_4$ and interchanges $\Omega_1$ and $\Omega_2$.
\item[(c1)] $G$ has no short orbits preserved by $\Phi_q$.
\end{itemize}
Set $\overline{P}_1:=\pi_1(\Omega_1)$, $\overline{P}_2:=\pi_1(\Omega_2)$, $\overline{Q}_1:=\pi_2(\Omega_3)$ and $\overline{Q}_2:=\pi_2(\Omega_4)$. Suppose that (a1) holds. Then $\overline{P}_1, \overline{P}_2 \in \xx/C_n$ and $\overline{Q}_1, \overline{Q}_2 \in \xx/C_m$ are $\fq$-rational points. Hence by Lemma \ref{tec} we obtain (a). In case (b1.1) holds, we have that $\overline{P}_1, \overline{P}_2 \in \xx/C_n$ are $\fq$-rational, thus $m|q-1$. Furthermore, since $\overline{Q}_1, \overline{Q}_2 \in \xx/C_m$ are not $\fq$-rational, it follows from Lemma \ref{tec} that $n|q+1$. Similarly, (b1.2) implies that $n|q-1$ and $m|q+1$, thus we obtain (b). In view of the previous cases, (c) also follows from Lemma \ref{tec}.
\end{proof}

\begin{prop}\label{eq t4 1}
Let $\xx$ be a curve defined over $\fq$, where $q=p^h$ ($p>2$), satisfying (P). Assume that $G=C_n \times C_m$ has four short orbits in $\xx$, each of them preserved by $\Phi_q$. Then $\xx$ is $\fq$-birationally equivalent to a curve defined by $aX^nY^m+bX^n+cY^m=1$, with $a,b,c \in \fq$  with $c \neq \frac{a}{b}$ and $a \neq 0$.
\end{prop}
\begin{proof}
By Lemma \ref{so t4}(a), both $n$ and $m$ divide $q-1$, and the points $\overline{P}_1, \overline{P}_2, \overline{Q}_1$ and $\overline{Q}_2$ are $\fq$-rational. Hence $\fq(\xx/C_n)=\fq(y)$, $\fq(\xx/C_m)=\fq(x)$ and $\fq(\xx/G)=\fq(y^m)=\fq(x^n)$, with $x,y \in \fq(\xx)$, by Proposition \ref{n|q-1}. Hence $$y^m=\frac{\alpha x^n+\beta}{\gamma x^n+ \eta},$$
where $\alpha,\beta,\gamma, \eta \in \fq$ such that $\alpha \eta \neq \beta \gamma$. With notation as in Lemma \ref{so t4}, the points $P_1,P_2,Q_1$ and $Q_2$ are pairwise distinct. Here, $\{P_1,P_2\}$ is the set of zero and pole of $y^m$, and $\{Q_1,Q_2\}$ the set of zero and pole of $x^n$. Therefore, $\beta$ and $\gamma$ are nonzero. Thus, we have obtained an irreducible equation. The result now follows from Lemma \ref{fiberp}.
\end{proof}

\begin{prop}\label{eq t4 2}
Let $\xx$ be a curve defined over $\fq$ satisfying (P). Assume that case $(b)$ in Lemma \ref{so t4} holds. Then $\xx$ is $\fq$-birationally equivalent to  curve defined by an affine equation
$$
\frac{aY^m+b}{cY^m+d} = \frac{i[(X+i)^n-(X-i)^n]}{(X+i)^n+(X-i)^n},
$$
with $a,b,c,d \in \fq$ such that $ad\neq bc$ and $u+iv \in \mathbb{F}_{q^2}$ is an $2n$-th root of unity.
\end{prop}
\begin{proof}
Without loss of generality, we can assume that $m | q-1$ and $n | q+1$, with the points $\overline{P}_1, \overline{P}_2 \in \xx/C_n$ being $\fq$-rational. The result then follows from Propositions \ref{n|q-1}, \ref{ffq+1} and Lemmas \ref{fiberp} and  \ref{fact}.
\end{proof}
\begin{prop}\label{eq t4 3}
Let $\xx$ be a curve defined over $\fq$ satisfying $(P)$. Assume that $(c)$ in Lemma \ref{so t4} holds. Then $\xx$ is $\fq$-birationally equivalent to a curve defined by an affine equation

$$
\frac{[(ai+b)(X-i)^n+(b-ai)(X+i)^n][(Y-i)^m+(Y+i)^m]}{ i[(ci+d)(X-i)^n+(d-bi)(X+i)^n][(Y-i)^m-(Y+i)^m]} = 1,
$$
with $a,b,c,d \in \fq$, with $ad \neq bc$.
\end{prop}
\begin{proof}
Arguing as in the previous cases, the result follows from Proposition \ref{ffq+1} and Lemmas \ref{fiberp} and \ref{fact}.
\end{proof}
\noindent
{\bf Proof of Theorem \ref{main}}
It follows by collecting the results from Propositions \ref{genus}, \ref{eq t3}, Lemma \ref{so t4}, Propositions \ref{eq t4 1}, \ref{eq t4 2} and \ref{eq t4 3}. \qed

\begin{rem}\label{quadrex}
Let $\xx$ be a curve satisfying (P). If we want to characterize $\xx$ up to birational equivalence over $\kk$, then we have that either $\xx$ is $\F_{q^2}$-birationally equivalent to the curve defined by (\ref{t=3}) or to the one defined by (\ref{t=4 1}) (with the coefficients in $\F_{q^2}$). Indeed, since the short orbits of $G$ are preserved by $\Phi_{q^2}$, then $m$ and $n$ divide $q\pm 1$(Lemma \ref{tec}), and so Proposition \ref{n|q-1} applies to both extensions $\F_{q^2}(\xx/C_n):\F_{q^2}(\xx/G)$ and $\F_{q^2}(\xx/C_m):\F_{q^2}(\xx/G)$.
\end{rem}

\begin{rem}\label{overlap}
It is possible that a curve defined by (\ref{t=4 1}) admits a model given by an equation of type (\ref{t=3}) for distinct powers of $X$ and $Y$. Indeed, consider the hyperelliptic curve $\ff$ defined by the equation $X^nY^2+X^n+Y^2=1$. Then $\ff$ is $\fq$-birationally equivalent to the curve defined by $\bar{X}^{2n}+\bar{Y}^2=1$ via $(X,Y) \mapsto (\bar{X},\bar{Y}):=\Big(X,\frac{2Y}{Y^2+1}\Big)$.
\end{rem}

%%%%%%%%%%%%%%%%%%%%%%%%%%%%%%%%%%%%%%%%%%%%%%%%%%%%%%%%%%%%%%%%%%%%%%%%%%%%%%%%%%%%%%%%%%%%%%%%%%%%%%%%%%%%%%%%%%%%%%%%%%%%%%%%%%%%%%%%%%%%%%%%%%%%%%%%%%%%%%%%%%%%%%%%%%%%%%%%%%%%%%

\section{The full automorphism group}\label{full}

In this section, we exploit the full automorphism group of a curve $\xx$ defined over $\fq$ satisfying (P). Recall that $\kk$ denotes the algebraic closure of $\fq$, where $q=p^h$.  
According to Theorem \ref{main} and Remark \ref{quadrex}, one of the following holds:
\begin{itemize}
\item [(I)] $G=C_n \times C_m$ has $3$ short orbits on $\xx$, and $\xx$ is $\kk$-birationally equivalent to the curve defined by $X^n+Y^m=1$.
\item[(II)] $G=C_n \times C_m$ has $4$ short orbits on $\xx$, and $\xx$ is $\kk$-birationaly equivalent to the curve defined by $aX^nY^m+X^n+Y^m=1$, with $a \in \kk^{*}$.
\end{itemize}

The full automorphism group $\aut_\kk(\xx)$ of the curves $\xx$ of case (I) above is completely characterized by Kontogeorgis \cite{Ko}, provided that $p>3$, $m \neq n$, $n \neq 4$ and $m \neq 3$. We summarize such characterization in the following Theorem.

\begin{thm}[Kontogeorgis]\label{konto2}
Let $\xx$ be a nonsingular model of the curve given by the affine equation $X^n+Y^m=1$, where $m < n$ with $(m,n) \neq (3,4)$. Then $C_m$ is a normal subgroup of $\aut_{\kk}(\xx)$, and
\begin{equation}\nonumber
\aut_{\kk}(\xx)/C_m \cong \begin{cases}
C_{n}, \ \text{if $m \nmid n$; }\\
D_{n}, \ \text{if $m | n$ but $n-1$ is not a power of $p$; }\\
\PGL(2,p^r), \ \text{if $m | n$ and $n-1=p^r$ for some $r>0$. }
\end{cases}
\end{equation}
\end{thm}
If $m=n$, the case (I) provides the Fermat curve $X^n+Y^n=1$. In this situation, it is well known that if $n=p^r+1$ for some $r>0$, then $\aut_{\kk}(\xx)\cong \PGU(3,p^r)$ (see e.g. \cite[Proposition 11.30]{HKT}) and if  $n \neq p^r+1$ for all $r>0$, then $C_n \times C_n$ is normal in $\aut_{\kk}(\xx)$ and $\aut_{\kk}(\xx)/(C_n \times C_n) \cong S_3$ (see \cite[Theorem 11.31]{HKT}).

In view of such characterizations, in what follows in this section we assume that (II) above holds. Our main goal is to characterize $\aut_\kk(\ff)$, where $\ff:aX^nY^m+X^n+Y^m=1$, where $\max\{n,m\}>2$. Let $\zeta_1, \zeta_2, c_1, c_2 \in \kk$ such that $\zeta_1$ (resp. $\zeta_2$) is an $n$-th (resp. $m$-th) primitive root of the unity, and $c_1^n=c_2^m=-a^{-1}$. Let $x$ and $y$ such that $\kk(\ff)=\kk(x,y)$ and $ax^ny^m+x^n+y^m=1$. From the equation of $\ff$, one can see that $\aut_{\kk}(\ff)$ contains the following elements:
$$
\sigma_1:(x,y) \mapsto (\zeta_1 x,y), \ \ \sigma_2:(x,y) \mapsto (x,\zeta_2 y) \ \ \text{ and } \ \ \mu: (x,y) \mapsto \left(\frac{c_1}{x},\frac{c_2}{y} \right).
$$
Here, we have $C_n=\langle \sigma_1 \rangle$, $C_m=\langle \sigma_2 \rangle$ and $\mu$ is an involution that normalizes both $C_n$ and $C_m$. Thus these three automorphisms generate a subgroup $G \rtimes \langle \mu \rangle < \aut_\kk(\ff)$ of order $2mn$. Moreover, since $\mu \sigma_i \mu =\sigma_i^{-1}$ for $i=1,2$, we have that $C_k \rtimes \langle \mu \rangle \cong D_k$, where $D_k$ denotes a dihedral group of order $2k$, with $k \in \{m,n\}$. 

If $a=1$, we can choose $c_1$ (resp. $c_2$) as a $2n$-th (resp. $2m$-th) primitive root of the unity. Hence $\zeta_i=c_i^2$, for $i=1,2$, and $\aut_\kk(\ff)$ contains
$$
\tau_1:(x,y) \mapsto \left(c_1 x , \frac{\zeta_2}{y}\right) \ \ \ \text{ and } \ \ \ \tau_2:(x,y) \mapsto \left(\frac{\zeta_1}{x} , c_2 y\right).
$$
Note that $\tau_i^2=\sigma_i$ for $i=1,2$, but $\tau_1 \tau_2 \neq \tau_2 \tau_1$. Furthermore, if $m=n$ we have the following extra involution
$$
\theta: (x,y) \mapsto (y,x).
$$ 

For convenience, from now on, $\xx$ stands for a nonsingular model of $\ff$. In order to characterize the full automorphism group of $\xx$, we will use the following results.

\begin{lem}\label{orbdih}
The group $C_n \rtimes \langle \mu \rangle$ acts transitively on $\Omega_1 \cup \Omega_2$, and $C_m \rtimes \langle \mu \rangle$ acts transitively on $\Omega_3 \cup \Omega_4$. 
\end{lem}
\begin{proof}
Regarding $\kk(\xx)$ as a Kummer extension of the rational function fields $\kk(x)$ and $\kk(y)$, it can be seen that, up to re-labeling the indices,
$$
\div(x)=\sum_{P \in \Omega_3}P-\sum_{Q \in \Omega_4}Q \ \ \ \text{ and } \ \ \  \div(y)=\sum_{R \in \Omega_1}R-\sum_{S \in \Omega_2}S.
$$
More precisely, $\Omega_1$ corresponds to $\{(\zeta_1^k:0:1)\ \ | \ 0 \leq k \leq n-1\}$ and $\Omega_2$ consists of the points of $\xx$ centered at $(0:1:0) \in \ff$, while $\Omega_3$ corresponds to $\{(0:\zeta_2^s:1)\ \ | \ 0 \leq s \leq m-1\}$ and $\Omega_4$ is the set of points of $\xx$ centered at $(1:0:0) \in \ff$. Clearly $\sigma_1$ acts transitively on $\Omega_1$ and $\Omega_2$,  while $\mu$ sends a zero of $y$ on a pole of $y$, and vice-versa. Hence the first statement follows. The second is analogous.
\end{proof}

\begin{lem}\label{interchange}
Assume that there exists $\eta \in \aut_\kk(\xx)$ such that $\eta$ preserves the set of zeros and the set of poles of $x$ and interchanges the set of zeros with the set of poles of $y$. Then $a=1$.
\end{lem}
\begin{proof}
Since $\eta$ preserves the set of zeros and the set of poles of $x$, then $\div(\eta(x))=\div(x)$, which means that $\eta(x)=\alpha x$ for some $\alpha \in \kk^{*}$. In the same way, $\eta$ interchanging the set of zeros with the set of poles of $y$ gives that $\div(\eta(y))=\div(y^{-1})$, and so $\eta(y)=\frac{\beta}{y}$ for some $\beta \in \kk^{*}$. Therefore, via the equation $a\eta(x)^n \eta(y)^m+\eta(x)^n+\eta(y)^m=1$, we obtain
$
(\alpha^n)x^ny^m+(a\alpha^n\beta^m)x^n-y^m+\beta^m=0,
$
which leads us to $a=1$, $\alpha^n=-1$ and $\beta^m=1$.
\end{proof}

\subsection{The case $n \neq m$}

We start our investigation with the case $m \neq n$. So, without loss of generality, assume that $m<n$.  Following notation of the previous section, we know that $G$ has $4$ short orbits on $\xx$, namely $\Omega_1$, $\Omega_2$, $\Omega_3$ and $\Omega_4$, where $\#(\Omega_1)=\#(\Omega_2)=n$ and $\#(\Omega_3)=\#(\Omega_4)=m$. Moreover, $\Omega_1 \cup \Omega_2$ is the precise set of fixed points of $C_m$ and  $\Omega_3 \cup \Omega_4$ is the precise set of fixed points of $C_n$. We begin from the following result. 

\begin{lem}\label{cmnormal}
Assume that $m<n$. Then $C_m$ is a normal subgroup of $\aut_{\kk}(\xx)$.
\end{lem}
\begin{proof}
It follows directly from the above discussion and Proposition \ref{konto1}.
\end{proof}

Let $f \geq h$ be the smallest integer such that $\aut_{\kk}(\xx)=\aut_{\F_{p^f}}(\xx)$. From Lemma \ref{cmnormal}, the full automorphism group of the quotient curve $\xx/C_m$ has  a subgroup $H$ defined over $\F_{p^f}$ isomorphic to $\aut_{\kk}(\xx)/C_m$. Since $\xx/C_m$ is rational and defined over $\F_{p^h}$, we have $\aut_{\F_{p^f}}(\xx/C_m) \cong \PGL(2,p^f)$, and thus $H$ has to be isomorphic to one of the groups in \cite[Theorem 3]{VM}. Recall that $\aut_{\kk}(\xx)$ has always $C_n \rtimes \langle \mu \rangle \cong D_n$ as a subgroup. Since such group meets $C_m$ trivially, we conclude that $H$ has a subgroup isomorphic to $D_n$. With this on hands, we are able to prove the following.

\begin{thm}\label{fullpar}
Assume that $m<n$, with $n \neq 4$. If $\aut_{\kk}(\xx)$ has no $p$-subgroup, then
\begin{equation}\nonumber
\aut_{\kk}(\xx)/C_m \cong \begin{cases}
D_{2n}, \ \text{if $a = 1$, }\\
D_{n}, \ \text{if $a \neq 1$. } 
\end{cases}
\end{equation}
\end{thm}
\begin{proof}
Since $D_n<  \aut_{\kk}(\xx)/C_m$, by \cite[Theorem 3]{VM}  $\aut_{\kk}(\xx)/C_m$ is isomorphic to one of the following groups: $D_\ell$ with $\ell | p^f \pm 1$, $S_4$ and $A_5$ (the remaining groups in the Hauptsatz \cite[Theorem 3]{VM} Note that, since the point-set $\Omega_1 \cup \Omega_2 \subset \xx$ is the precise set of fixed points of $C_m$, Lemma \ref{orbdih}  gives that $\Omega_1 \cup \Omega_2$ is a short orbit of $\aut_{\kk}(\xx)$. Assume that $\aut_{\kk}(\xx)/C_m \cong A_5$. Then, since $D_5$ is the unique dihedral subgroup of $A_5$, we obtain $n=5$. Moreover, $|\aut_{\kk}(\xx)|=60m$, and by the Orbit-Stabilizer Theorem, the stabilizer of a point $P \in \Omega_1 \cup \Omega_2$ in $\aut_{\kk}(\xx)$ has size $6m$. Hence, the stabilizer of $\pi_2(P) \in \xx/C_m$ in $\aut_{\kk}(\xx)/C_m$ has size $6$. Thus, \cite[Theorem 1]{VM} provides that $\aut_\kk(\xx)$ has only two short orbits: one of size $10$ and one of size $12m$. Then, by Riemman-Hurwitz formula (\ref{rhso}) applied to the cover $\xx \rightarrow \xx/\aut_\kk(\xx)$, we obtain
%$$
%8(m-1)-2=-120m+(60m-10)+(60m-12m),
%$$
%which gives 
$m=0$, a contradiction. Thus $\aut_{\kk}(\xx)/C_m \not\cong A_5$. We also have $\aut_{\kk}(\xx)/C_m \not\cong S_4$, from $D_4$ being the unique dihedral subgroup of $S_4$, which implies $n=4$.

Therefore, $\aut_{\kk}(\xx)/C_m \cong D_\ell$ for some $\ell$. From the normality of $C_m$, there exists $\bar{C}_n < \aut_{\kk}(\xx)/C_m$ such that $\bar{C}_n \cong C_n$. In particular, $n | \ell$. Since $\Omega_3 \cup  \Omega_4$ is pointwise fixed by $C_n$, both points $\pi_2(\Omega_3), \pi_2(\Omega_4) \in \xx/C_m$ are fixed by $\bar{C}_n$. Thus from \cite[Theorem 1]{VM} it follows that all the points of $(\xx/C_m) \backslash\{\pi_2(\Omega_3), \pi_2(\Omega_4)\}$ are in long orbits of $\bar{C}_\ell$, where $\bar{C}_\ell$ is the cyclic normal subgroup of $D_\ell$. Hence the set $\pi_2(\Omega_1 \cup \Omega_2)$ is an union of long orbits of $\bar{C}_\ell$. Therefore, since $\#(\pi_2(\Omega_1 \cup \Omega_2))=\#(\Omega_1 \cup \Omega_2)=2n$, we conclude that $\aut_{\kk}(\xx)/C_m \cong D_\ell$, with $\ell \in \{n,2n\}$. We will proceed to prove that $\ell=2n$ if and only if $a=1$. If $a=1$, we have defined on $\xx$ the automorphism $\tau_1:(x,y) \mapsto \left(c_1 x , \frac{\zeta_2}{y}\right)$.
The group $\langle \tau_1 \rangle$ is cyclic of order $2n$, and it meets $C_m$ trivially.  Thus $\langle \tau_1 \rangle \cong \bar{C}_\ell$ and $\ell=2n$. Assume now that $\ell=2n$. As before, denote by $\bar{C}_\ell$ the cyclic subgroup of $D_\ell$ of order $\ell$. Let $\bar{\delta}$ such that $\langle \bar{\delta} \rangle =\bar{C}_\ell$ and consider $\delta \in \aut_{\kk}(\xx)$ such that $\bar{\delta}$ is the image of $\delta$ under the natural group projection of $\aut_{\kk}(\xx)$ onto $\aut_{\kk}(\xx)/C_m$. On one hand, since $\bar{\delta}$ acts transitively on $\pi_2(\Omega_1 \cup \Omega_2)$ and $\bar{\delta}^2$ acts transitively on both $\pi_2(\Omega_1)$ and $\pi_2(\Omega_2)$, we see that $\bar{\delta}$ gives an injection from $\pi_2(\Omega_1)$ onto $\pi_2(\Omega_2)$. Thus $\delta$ maps bijectively the set of zeros onto the set of poles of $y$. On the other hand, $\bar{\delta}$ fixes both $\pi_2(\Omega_3), \pi_2(\Omega_4) \in \xx/C_m$, whence $\delta$ preserves both $\Omega_3$ and $\Omega_4$. In other words, $\delta$ preserves the set of zeros and the set of poles of $x$. Hence the result follows from Lemma \ref{interchange}.
\end{proof}

Now we study the case in which $\aut_{\kk}(\xx)$ has a $p$-subgroup. We start by pointing out that $\xx$ can have automorphisms of order $p$.

\begin{lem}\label{hyper}
Let $\xx$ be a nonsingular model of the hyperelliptic curve defined over $\kk$ by the equation $X^nY^2+X^n+Y^2=1$, where $n=\frac{p^r+1}{2}$ for some $r>0$. Then $\aut_{\kk}(\xx)/C_2 \cong \PGL(2,p^r)$.
\end{lem}
\begin{proof}
Remark \ref{overlap} implies that $\xx$ has a plane model defined by $\bar{X}^{p^r+1}+\bar{Y}^2=1$. Therefore, the result follows from Theorem \ref{konto2}. 
\end{proof}

\begin{lem}\label{p-group}
Assume that $m<n$ and that $\aut_{\kk}(\xx)$ has a $p$-subgroup of order $p^r$, with $r>0$. Then such $p$-group has a single fixed point $P \in \xx$ such that $P \in \Omega_1 \cup \Omega_2$, and it acts semi-regularly on the remaining points of $\xx$. Furthermore,  $p^r | 2n-1$.
\end{lem}
\begin{proof}
Assume that $\aut_{\kk}(\xx)$ has a subgroup $E$ of order $p^r$. Since $p \nmid m$, we have that $E$ meets $C_m$ trivially, and thus  by \cite[Theorem 3]{VM} there is an elementary abelian $p$-group $E_{p^r}<\aut_{\kk}(\xx)/C_m$ isomorphic to $E$ (in particular, $E$ must be abelian). From \cite[Theorem 1]{VM}, the group $E_{p^r}$ fixes only one point and acts semi-regularly on the remaining points of $\xx/C_m$. By the normality of $C_m$, we see that $E_{p^r}$ acts on the points of $\xx/C_m$ as $E$ does on the set of orbits of $C_m$ on $\xx$. Since $\Omega_1 \cup \Omega_2$ is the precise set of fixed points of $C_m$, this means that $E(\Omega_1 \cup \Omega_2)=\Omega_1 \cup \Omega_2$, and thus $E_{p^r}(\Gamma)=\Gamma$, where $\Gamma= \pi_2(\Omega_1 \cup \Omega_2)$. So let $Q \in \xx/C_m$ be the unique fixed point of $E_{p^r}$. If $Q \notin \Gamma$, then $\Gamma$ would be a union of long orbits of $E_{p^r}$, whence $p^r | \#(\Gamma)=2n$, a contradiction. Therefore, $Q \in \Gamma$, whence $P:=\pi_2^{-1}(Q) \in \Omega_1 \cup \Omega_2$ is the only fixed point of $E$. In addition, $\Gamma \backslash \{Q\}$ is a union of long orbits of $E_{p^r}$, which finishes the proof. 
\end{proof}

\begin{prop}\label{2n=p^r+1}
Assume that $m<n$ and that $\aut_{\kk}(\xx)$ has a Sylow $p$-subgroup of order $p^r$, with $r>0$. Then $\PSL(2,p^r)< \aut_{\kk}(\xx)/C_m$ and $n=\frac{p^r+1}{2}$. 
\end{prop}
\begin{proof}
From \cite[Theorem 3]{VM}, $\aut_{\kk}(\xx)/C_m$ is isomorphic either to $\PSL(2,p^r)$ or to $\PGL(2,p^r)$, with $r|f$. In any case, $\PSL(2,p^r)< \aut_{\kk}(\xx)/C_m$. So, on one hand, since $p \nmid m$, we have that $\aut_{\kk}(\xx)$ has a subgroup $E \cong E_{p^r}$. Thus Lemma \ref{p-group} gives us that $p^r|2n-1$. On the other hand, \cite[Theorem 3]{VM} implies that $n | p^r \pm 1$. If $n | p^r - 1$, we would have integers $s_1,s_2>0$ such that $2n-1=s_1p^r$ and $p^r-1=s_2n$. Then $(2-s_1s_2)p^r=s_2+2$, which is only possible for $s_1=s_2=1$ and $p^r=3$, and so $n=2$ and $m=1$, a contradiction. If $n | p^r + 1$, as in the previous case, write $2n-1=s_1p^r$ and $p^r+1=s_2n$, with integers $s_1,s_2 >0$.  Thus $(2-s_1s_2)p^r=s_2-2$, which is only possible for $s_1=1$ and $s_2=2$. Then $n=\frac{p^r+1}{2}$.
\end{proof}

\begin{prop}\label{pgrouprat}
Assume that $m<n$ and that $\aut_{\kk}(\xx)$ has a Sylow $p$-subgroup $E$ of order $p^r$, with $r>0$. Then the quotient curve $\xx/E$ is rational.
\end{prop}
\begin{proof}
By Proposition \ref{2n=p^r+1}, there exists $W<\aut_{\kk}(\xx)$ such that $W/C_m \cong \PSL(2,p^r)$. Moreover, since $n=\frac{p^r+1}{2}$, one can check that $C_n<W$. From \cite[Theorem 1]{VM}, $\PSL(2,p^r)$ has two short orbits on $\xx/C_m$, a non-tame one, which we denote by $\Gamma_1$, and a tame one, denoted by $\Gamma_2$. Moreover, $\#(\Gamma_1)=p^r+1=2n$ and $\#(\Gamma_2)=p^r(p^r-1)$.  Set $\Lambda_1=\Omega_1 \cup \Omega_2 \subset \xx$. Each point of $\Lambda_1$ is fully ramified in the cover $\xx \rightarrow \xx/C_m$, and the remaining points of $\xx$ have trivial stabilizer in $C_m$. By Lemma \ref{p-group}, there is a point $P_1 \in \Lambda_1$ fixed by $E$, and $E$ acts transitively on $\xx \backslash \{P_1\}$; in particular, $E=W_{P_1}^{(1)}$. Since $C_m < W_{P_1}$ and $p \nmid m$, the point $\pi_2(P_1) \in \xx/C_m$ has a non-tame stabilizer. Thus $\pi_2(P_1) \in \Gamma_1$. By the normality of $C_m$, $\PSL(2,p^r)$ acts on the points of $\xx/C_m$ as $W$ does on the set of orbits of $\xx$. Hence, $\Lambda_1$ is a short orbit of size $2n$ of $W$. Now let $Q_1 \in \Omega_3 \cup \Omega_4$. Since $C_n$ fixes $\Omega_3 \cup \Omega_4$ pointwise, we have that $\pi_2(\Omega_3) \in \xx/C_m$ has nontrivial stabilizer. Since $\pi_2(\Omega_3) \notin \Gamma_1$, we have that $\pi_2(\Omega_3) \in \Gamma_2$, and so its stabilizer on $\PSL(2,p^r)$ is tame, and by \cite[Theorem 1]{VM} it has size $n$. The same holds for $\pi_2(\Omega_4)$. Therefore, from the fact that every point of $\xx$ outside $\Lambda_1$ has trivial stabilizer on $C_m$, we see that $W$ has another short orbit $\Lambda_2$ (containing $\Omega_3 \cup \Omega_4$) on $\xx$ of size $mp^r(p^r-1)$, and $\Lambda_1$ and $\Lambda_2$ are the only short orbits of $W$. Thus, recalling that $\xx/C_m$ is rational and $|\PSL(2,p^r)|=p^r(p^r-1)n$, the Riemann-Hurwitz formula (\ref{rhg}) applied to the cover $\xx \rightarrow \xx/W$ gives
$$
2mn-2m-2n= -2mnp^r(p^r-1)+mp^r(p^r-1)(|W_{Q_1}|-1)+2n(|W_{P_1}|-1 )+2n \sum_{i \geq 1}\big(|W_{P_1}^{(i)}|-1  \big),
$$
and so
$$
2mn-2m-2n= -2mnp^r(p^r-1)+mp^r(p^r-1)(n-1)+2n\left(\frac{m(p^r-1)p^r}{2}+p^r-2\right)+2n \sum_{i \geq 2}\big(|W_{P_1}^{(i)}|-1  \big),
$$
which gives 
\begin{equation}\label{equa1}
\sum_{i \geq 2}\big(|W_{P_1}^{(i)}|-1  \big)=(m-1)(p^r-1). 
\end{equation}
Now, denote by $\tilde{g}$ the genus of $\xx/E$. From $\{P_1\}$ being the unique short orbit of $E$, the Riemann-Hurwitz formula (\ref{rhg}) applied to $\xx \rightarrow \xx/E$ provides
$$
2mn-2m-2n=2p^r(\tilde{g}-1)+2p^r-2+ \sum_{i \geq 2}\big(|E_{P_1}^{(i)}|-1  \big).
$$
Hence
\begin{equation}\label{equa2}
\tilde{g}=\frac{(m-1)(p^r-1)-\sum_{i \geq 2}\big(|E_{P_1}^{(i)}|-1  \big)}{2p^r}.
\end{equation}

The result now follows from (\ref{equa1}), (\ref{equa2}) and from the equality $E=W_{P_1}^{(1)}$.
\end{proof}

Let $x,y \in \kk(\xx)$ such that $\kk(x,y)= \kk(\xx)$ and $ax^ny^m+x^n+y^m=1$.  If $\aut_{\kk}(\xx)$ contains a Sylow $p$-subgroup $E$ of order $p^r$, then $\xx/E$ is rational and $\xx$ has a point $P_1$ that is the unique point fixed by $E$. Moreover,  $P_1 \in \Omega_1 \cup \Omega_2$, and so there is  $\lambda \in \kk^{*}$ such that $\div((x-\lambda)^{-1})_{\infty}=mP_1$. Furthermore, thanks to the proof of Theorem \ref{pgrouprat}, we know that $\aut_{\kk}(\xx)$ has no fixed points. Finally, $2n=p^r+1$ and $g=(n-1)(m-1)$. Therefore, via \cite[Theorem 12.4 and Theorem 12.11]{HKT}, we have the following.

\begin{prop}\label{sep}
Assume that $m<n$ and that $\aut_{\kk}(\xx)$ has a Sylow $p$-subgroup $E$ of order $p^r$, with $r>0$. Then:
\begin{itemize}
\item $\kk(\xx)=\kk(z,w)$, where $z^{p^r}+z=w^m$, with $p^r \equiv -1 (\mod m)$ and $P_1$ is the common pole of $z$ and $w$;
\item $\aut_{\kk}(\xx)/C_m \cong \PGL(2,p^r)$;
\item $C_m$ fixes each of the $p^r+1$ points with the same Weierstrass semigroup as $P_1$;
\item $\aut_{\kk}(\xx)/C_m$ acts on the set of such $p^r+1$ points as $\PGL(2,p^r)$.
\end{itemize}
\end{prop}

\begin{thm}\label{mainpgroup}
Assume that $m<n$ and that $\aut_{\kk}(\xx)$ has a Sylow $p$-subgroup of order $p^r$. Then $\aut_{\kk}(\xx)/C_m \cong \PGL(2,p^r)$, $a=1$, $n=\frac{p^r+1}{2}$ and $m=2$.
\end{thm}
\begin{proof}
The first statement follows from Proposition \ref{sep}. We keep the notation of the proof of Proposition \ref{pgrouprat}. Arguing as in the proof of such Proposition, we see that $\aut_{\kk}(\xx)$ has only two short orbits: $\Lambda_1=\Omega_1 \cup \Omega_2$, that has size $2n$, and $\Lambda_2$, of size $mp^r(p^r-1)$. Let $Q_1 \in \Omega_3 \cup \Omega_4 \subset \Lambda_2$. From Proposition \ref{sep}, we obtain that the stabilizer $(\aut_{\kk}(\xx))_{Q_1}$ of $Q_1$ is a cyclic group of order $2n$ that acts transitively on $\Lambda_1$. Thus Lemma \ref{interchange} implies that $a=1$. Moreover, $(\aut_{\kk}(\xx))_{Q_1}= \langle \tau_1 \rangle$. The group $\langle \tau_1 \rangle$ is, up to conjugacy, the only cyclic group of order $2n$ fixing a point of $\xx$. This happens because no such group is in the stabilizer $(\aut_{\kk}(\xx))_P$ of some point $P \in \Lambda_1$, since $|(\aut_{\kk}(\xx))_P|=mp^r(p^r-1)=2m(2n-1)(n-1)$. Hence every cyclic group of order $2n$ fixing a point is the stabilizer of some point $Q \in \Lambda_2$, whence conjugated to $\langle \tau_1 \rangle$. Furthermore, $C_m$ is the only cyclic normal subgroup of $\aut_{\kk}(\xx)$ of order $\geq m$. Indeed, the existence of a cyclic normal subgroup $T \neq C_m$ of $\aut_{\kk}(\xx)$ would imply the existence of a cyclic normal subgroup $\bar{T}$ of $\aut_{\kk}(\xx)/C_m \cong \PGL(2,p^r)$, a contradiction.

Now let $\varepsilon, \kappa \in \kk$ such that $\varepsilon^{p^r+1}=-1$ and $\kappa^{p^r}+\kappa=1$. Then $\kk(u,v)=\kk(z,w)$, where $u=\frac{\varepsilon}{z-\kappa}+\varepsilon$ and $v=w\left(\frac{\varepsilon}{z-\kappa}\right)^{\frac{p^r+1}{m}}$. Note that $v^{m}=u^{p^r+1}+1$.
In other words, $\xx$ is birationally equivalent to the curve $\cc$ defined by the affine equation $X_1^{2n}+X_2^{m}=1$. Denote by $K_{2n}$ the group generated by the automorphism $\phi_1:(u,v) \mapsto (c_1u,v)$ and by $K_m$ the group generated by $\phi_2:(u,v) \mapsto (u, \zeta_2 v)$. It should be noted that $\phi_1$ and $\phi_2$ commute. The group $K_m$ fixes $2n$ points of $\xx$ (this was shown in the proof of Lemma \ref{so t3}, since $K_{2n} \times K_m$ has three short orbits on $\xx$). Thus $K_m$ is a normal subgroup of $\aut_{\kk}(\xx)$, by Theorem \ref{konto1}. Therefore, $K_m=C_m$. On its turn, $K_{2n}$ also has fixed points on $\xx$, by the same reasons that $K_m$ does. Hence $K_{2n}=\sigma C_{2n} \sigma^{-1}$ for some $\sigma \in \aut_{\kk}(\xx)$. But $C_m$ is a normal subgroup of $\aut_{\kk}(\xx)$ and $\sigma_2$ commutes with $\phi_1$. Thus $\sigma_2$ commutes with $\tau_1$. In particular, $\tau_1 \sigma_2 (y)=\frac{\zeta_2^2}{y}$ and $\sigma_2 \tau_1(y)= \frac{1}{y}$. Therefore $\zeta_2^2=1$, which means that $m=2$. This finishes the proof. 
\end{proof}

\begin{thm}
Let $\xx$ be a nonsingular model of the curve defined over the algebraic closure $\kk$ of $\fq$ by the equation $aX^nY^m+X^n+Y^m=1$, where $m<n$ and $p \nmid mn$. Assume that $n \neq 4$. Then $\aut_{\kk}(\xx)$ has a normal cyclic subgroup $C_m$ of order $m$, and
\begin{equation} %\nonumber
\aut_{\kk}(\xx)/C_m \cong \begin{cases}
D_{n}, \ \text{if $a \neq 1$; }\\
D_{2n}, \ \text{if $a=1$ and $(m,n) \neq \left(2,\frac{p^r+1}{2}\right)$  for all $r>0$; }\\
\PGL(2,p^r) \ \text{if $a=1$ and $(m,n) = \left(2,\frac{p^r+1}{2}\right)$ for some $r>0$.}
\end{cases}
\end{equation}
\end{thm}
\begin{proof}
This follows from Theorems \ref{fullpar} and \ref{mainpgroup}.
\end{proof}

\begin{rem}
Let  $ H =  \aut_{\kk}(\xx)/C_m$ and $s =  |\aut_{\kk}(\xx)/C_m|$ . If $\gcd(m,s) = 1$, then  $\aut_{\kk}(\xx) \cong C_m \rtimes H$ by the Schur-Zassenhaus Theorem. If $\gcd(m,s) \neq 1$, it might happen that $C_m$ has no complement in $\aut_{\kk}(\xx)$.
\end{rem}

\subsection{The case $n = m$}

If $n = m$, a different approach for the determination of the full automorphism group of $\xx$ is needed. Henceforth, $\xx$ is a nonsingular model of $\ff:aX^mY^m+X^m+Y^m=1$, where $a \in \kk^{*}$.

\begin{lem}
The set $\Omega = \Omega_1 \cup \Omega_2 \cup \Omega_3 \cup \Omega_4$ is an $\aut_{\kk}(\xx)$-short orbit of size $4m
$.
\end{lem}

\begin{proof}
First, note that $\Omega$ is a unique orbit under the action of the group generated by $C_m\times C_m$, $\theta$ and $\mu$ since  $\theta(\Omega_1) = \Omega_3$ and $\theta(\Omega_2) = \Omega_4$ while $\mu(\Omega_1) = \Omega_2$ and $\mu(\Omega_3) = \Omega_4$. Also, it is immediately seen that $m$ is a non-gap at any point lying on $\Omega$. Next, we show that $m$ is a gap number  at any point  $O \not\in \Omega$, is centered at $U= (b:c:1) \in \ff$ with $ b, c \neq 0$. Let $\ell$ be the tangent line to $\ff$  at $U$. It can be straightforwardly checked that $U$ is not an inflection point of $\ff$, from where the intersection multiplicity $I( U, \ff \cap \ell)= 2$. Then the curve $\mathcal{C}$ having the vertical line $X- a$ counted $m-3$ times, the line $Z=0$ counted $m-1$ times and $\ell$ as components is a canonical adjoint for $\xx$ such that  $I( U, \ff \cap \mathcal{C})= m-1$. This implies that $m$ is a gap number at $O$. 
\end{proof}

\begin{lem}\label{linearseries}
Let $\Omega_1 = \{ R_1,\ldots, R_m\}, \Omega_2 = \{ S_1,\ldots, S_m\}, \Omega_3 = \{ P_1,\ldots, P_m\}, \Omega_4 = \{ Q_1,\ldots, Q_m\}$. For $i,j,h,k$ such that $i+j+h+k=m$, consider the divisor
$$
D:=P_1+\cdots +P_i + Q_1+\cdots +Q_j + R_1+ \cdots + R_h + S_1+\cdots +S_k.
$$
Then the linear series $|D|$ has projective dimension 
$$
\dim |D| =  \begin{cases}
1, \ \text{ if $l=m$ for some $ l \in \{i,j, h,k\}$;}\\
0, \ \text{ otherwise}.
\end{cases}
$$ 
\end{lem}
\begin{proof}
For the first part of the assertion, without loss of generality we may assume $ i = m $, i.e. $D = P_1+\cdots+ P_m$. Recall that any linear series is cut out on $\ff$ by the adjoints of some degree $s$. Since $\ff$ has only ordinary $m$-fold singularities at $(0:1:0)$ and $(1:0:0)$ , a curve $\mathcal{C}$ of degree $s$ is an adjoint for $\ff$ if and only if $\mathcal{C}$ has at least an $(m-1)$-th fold point at each of these points. The degree $m$ curve $\mathcal{C}: XZ^{m-1} = 0$ is such that the intersection divisor $\mathcal{C} \cdot \ff = P_1+\ldots +P_m + m(S_1+\ldots + S_m) + (m-1)(Q_1+\ldots+ Q_m)$. Hence, the linear series $|D|$ is cut out on $\ff$ by all the curves of degree  $m$  intersecting $\ff$ at least $m$ times in each of the $S_h$'s and at least $(m-1)$ times at each of the $Q_j$'s. Since all such curves are of the type $a_1XZ^{m-1}+a_2Z^m=0$, our assertion follows. For the second part, without loss of generality we may assume that the support of $D$ is contained in $\Omega_1 \cup \Omega_3$. Then 
arguing as in the previous case, it is easily seen that the linear series $|D|$ in this case is cut out on $\ff$ by all the curves of degree $m+1$ passing through $P_{i+1},\ldots,  P_m$, $R_{j+1},\ldots,  R_m$ and at least $m-1$ times at each of the $Q_j$'s and $S_h$'s. Since there is just one curve satisfying such condition, namely $\mathcal{G}: XYZ^{m-1}=0$, our assertion follows. 
\end{proof}

\begin{prop}
$\aut_{\kk}(\xx)$ admits a representation as a permutation group on the set $\{\Omega_1, \Omega_2, \Omega_3, \Omega_4\}$. 
\end{prop}
\begin{proof} 
 By contradiction, assume there exists $\alpha \in \aut_{\kk}(\xx)$ which does not permute the $\Omega_i$'s. Let $D = P_1+\ldots+ P_m$. Then the support of $\alpha(D)$ is contained in more than one of the short orbits $\Omega_i$. In particular, $ 1 = \dim (|D|) = \dim (|\alpha(D)|)$, a contradiction by Lemma  \ref{linearseries}.
\end{proof}

\begin{thm}
If $n = m$, then $G=C_m \times C_m$ is normal in $\aut_{\kk}(\xx)$, and
$$
\aut_{\kk}(\xx)/G \cong \begin{cases}
C_2 \times C_2, \ \text{if $a\neq 1$;}\\
D_4, \ \text{if $a = 1$. }
\end{cases}
$$ 
\end{thm}
\begin{proof} Let $\alpha \in \aut(\xx)$ be such that $\alpha(\Omega_i) = \Omega_i$ for each $ i \in \{1,2,3,4\}$. Then $\alpha(x) = c_1x$ and $\alpha(y) = c_2y$. It is then straightforward to see that $c_1, c_2$  are $m$-th roots of the unity, whence $\alpha \in G$. This means that the kernel of the representation of $\aut_{\kk}(\xx)$ as permutation group on $4$ letters is $G$, which is hence a normal subgroup of $\aut_{\kk}(\xx)$. Also,  $|\aut_{\kk}(\xx)| \leq 24m^2$. We claim that there is no automorphism  fixing one of the short orbits and acting as a $3$-cycle on the other three. By contradiction, assume that such an automorphism $\alpha$ exists; without loss of generality, we may assume that $\alpha(\Omega_3) = \Omega_3$. Then either $\alpha(\Omega_4) = \Omega_1$ or $\alpha(\Omega_4) = \Omega_2$, whence  $\alpha(x)$ either belongs to the Riemann-Roch space $\mathcal{L}(R_1+\ldots + R_m)$ or $\mathcal{L}(S_1+\ldots +S_m)$. Then  $\alpha(x) = a_1 +\frac{b_1}{y}$ or $a_2 + b_2y$ by Lemma \ref{linearseries}. Via straightforward computations, one can see that neither of the latter function can have zeroes in $\Omega_3$. The discussion at the beginning of the section finishes the proof. 
\end{proof}

\begin{rem}
The results of subsections $6.1$ and $6.2$ imply that $\aut_{\kk}(\xx) = \aut_{\fqq}(\xx)$. 
\end{rem}

\begin{rem}
We saw in RemarK \ref{overlap} that the curves $X^nY^2+X^n+Y^2=1$ and $\bar{X}^{2n}+\bar{Y}^2=1$ are birationally equivalent. It is not difficult to show that this is the only case of overlap between curves of type (I) and (II), listed at the beginning of this section.
\end{rem}

%%%%%%%%%%%%%%%%%%%%%%%%%%%%%%%%%%%%%%%%%%%%%%%%%%%%%%%%%%%%%%%%%%%%%%%%%%%%%%%%%%%%%%%%%%%%%%%%%%%%%%%%%%%%%%%%%%%%%%%%%%%%%%%%%%%%%%%%%%%%%%%%%%%%%%%%%%%%%%%%%%%%%%%%%%%%%%%%%%%%%%%%

\subsection*{Acknowledgments}
The first author was supported by FAPESP-Brazil, grant 2013/00564-1. The second author was partially supported by GNSAGA - Gruppo Nazionale per le Strutture Algebriche, Geometriche e le loro Applicazioni of Italian INdAM.
The authors would also like to thank G\'abor Korchm\'aros and Massimo Giulietti for many useful conversations on the topic of this article.

\printindex

\end{document}